\documentclass[oneside,10pt,leqno]{amsart}
\usepackage{amssymb,amsmath,latexsym,amscd}
\usepackage{array,hhline}

\usepackage{longtable}

\def\gg{\mathfrak{g}}
\def\gh{\mathfrak{h}}

\def\gl{\mathfrak{l}}

\def\gn{\mathfrak{n}}

\def\gp{\mathfrak{p}}

\def\gs{\mathfrak{s}}
\def\gt{\mathfrak{t}}
\def\gu{\mathfrak{u}}
\def\gv{\mathfrak{v}}
\def\gw{\mathfrak{w}}

\def\gz{\mathfrak{z}}

\def\C{\mathbb{C}}

\def\E{\mathbb{E}}
\def\F{\mathbb{F}}

\def\H{\mathbb{H}}

\def\O{\mathbb{O}}

\def\R{\mathbb{R}}

\def\cH{\mathcal{H}}

\def\cO{\mathcal{O}}

\def\cS{\mathcal{S}}

\def\cU{\mathcal{U}}

\def\Im{{\rm Im}\,}
\def\Re{{\rm Re}\,}

\def\Ad{{\rm Ad}}

\def\Pf{{\rm Pf}}

\def\Ind{{\rm Ind\,}}

\def\tr{{\rm trace\,}}

\renewcommand{\thesection}{\arabic{section}}

\renewcommand{\thetable}{{\large \thesection.\arabic{equation}}}
\newtheorem{theorem}[equation]{Theorem}

\newtheorem{lemma}[equation]{Lemma}

\newtheorem{proposition}[equation]{Proposition}

\newtheorem{definition}[equation]{Definition}

\newtheorem{example}[equation]{Example}

\def\sideremark#1{\ifvmode\leavevmode\fi\vadjust{\vbox to0pt{\vss% the remark
 \hbox to 0pt{\hskip\hsize\hskip1em%                          will appear only
\vbox{\hsize2cm\tiny\raggedright\pretolerance10000 %          on the side
 \noindent #1\hfill}\hss}\vbox to8pt{\vfil}\vss}}} %          in 2cm
\begin{document}

\title{Unitary Representations, $L^2$ Dolbeault Cohomology, and 
Weakly Symmetric Pseudo--Riemannian Nilmanifolds}

\dedicatory{To the memory of Bert Kostant, a good friend and mathematical 
pioneer}

\author{Joseph A. Wolf}\thanks{Research partially supported by a Simons
Foundation grant}
\address{Department of Mathematics \\ University of California, Berkeley \\
	CA 94720--3840, U.S.A.} \email{jawolf@math.berkeley.edu}

\date{file ~/texdata/submitted/dolbeault-nilpotent/nil-flag.tex 
submitted 25 October 2018, some expository clarifications and typos 
corrected later, final edit 15 September 2019}



\begin{abstract}
We combine recent developments on weakly symmetric pseudo--riemannian
nilmanifolds with with geometric methods for construction of unitary
representations on square integrable Dolbeault cohomology spaces.  This runs
parallel to construction of discrete series representations on spaces
of square integrable harmonic forms with values in holomorphic vector
bundles over flag domains.  Some special cases had been described
by Satake in 1971 and the author in 1975.  Here we develop a 
theory of pseudo--riemannian nilmanifolds of complex type.  They can be
viewed as the nilmanifold versions of flag domains.  We construct the
associated square integrable (modulo the center) representations on
holomorphic cohomology spaces over those domains and note that
there are enough such representations for the Plancherel and Fourier 
Inversion Formulae there.  Finally, we note that the most interesting
such spaces are weakly symmetric pseudo--riemannian nilmanifolds, so
we discuss that theory and give classifications for three 
basic families of weakly symmetric pseudo--riemannian nilmanifolds of complex 
type.
\end{abstract}

\maketitle

\section{Introduction}\label{sec1} 
\setcounter{equation}{0}

This paper records and expands on a surprising observation.  It has long been 
known that the standard tempered representations of semisimple Lie groups 
--- which are enough for the Plancherel and Fourier Inversion Formulae 
there --- can be realized on partially holomorphic cohomology spaces over 
flag domains.  See \cite{W1973} and \cite{W2018}.  Here we develop a 
theory of pseudo--riemannian nilmanifolds of complex type, a sort of
nilmanifold version of flag domains.  We then construct the
associated square integrable (modulo the center) representations on
holomorphic cohomology spaces over those domains and note that
there are enough such representations for the Plancherel and Fourier 
Inversion Formulae there.  Finally, we note that many of the interesting
such spaces are weakly symmetric pseudo--riemannian nilmanifolds, so
we discuss that theory and give classifications for three 
basic families of weakly symmetric pseudo--riemannian nilmanifolds of complex 
type. 

The Bott--Borel--Theorem of the 1950's \cite{B1957} gave complex 
geometric realizations for representations of compact Lie groups.
In the early 1960's
Kirillov described the unitary dual for nilpotent Lie groups in terms of
coadjoint orbits \cite{K1962}.  Kostant saw the correspondence between 
those two theories and developed a common generalization, {\sl geometric 
quantization}.  In the framework of 
geometric quantization, the Bott--Borel--Theorem uses totally complex 
polarizations and the Kirillov theory uses real polarizations.
On the other hand, the infinite dimensional irreducible unitary 
representations of Heisenberg groups have complex realizations, on
spaces of Hermite polynomials.  

The extension of Bott--Borel--Weil to noncompact groups, perhaps inspired
by Harish-Chandra's holomorphic discrete series, was made plausible when
Andreotti and Vesentini \cite{AV1965} initiated the study of square
integrable Dolbeault cohomology.  The extension to noncompact 
real semisimple Lie groups, then called the Langlands Conjecture, was
the realization of square integrable (discrete series) representations of 
semisimple Lie groups on square integrable $\overline{\partial}$ cohomology 
(and certain variations) for holomorphic hermitian vector bundles 
$\E \to D$ over flag domains.  It was carried out by a number of people; 
see  Harish-Chandra \cite{HC1965}, Narasimhan and Okamoto \cite{NO1970},
Schmid \cite{Sc1971, S1976}, Wolf \cite{W1973, W1974} and 
Wong \cite{W1995, W1999}. 

Here we address the corresponding problem for a class of connected
unimodular Lie groups of the form $G = N \rtimes H$ where $N$ is a 
two-step nilpotent Lie group and $H$ is a closed reductive subgroup of 
$G$.  In the language of geometric quantization, we are looking for 
representations of $N$ defined by totally complex polarizations and
their extension to $G$.  The first example is the case where
$N$ is the Heisenberg group of dimension $2n+1$ and $H = U(n)$, or
more generally $U(p,q)$ with $p+q=n$.  There, the Fock representations of
$N$ extend to $G$ without Mackey obstruction, for purely geometric
reasons \cite{W1975}.  Also in \cite{W1975}, this  leads to the Plancherel 
and Fourier Inversion Formulae for $G = N \rtimes H$ and for certain
similar semidirect product groups. 

These Heisenberg group examples belong to a much larger family, the 
weakly symmetric riemannian (and pseudo--riemannian) nilmanifolds, studied 
in \cite{WC2018}.  Many members of that larger family enjoy special 
properties that combine complex geometry and real analysis.  For example 
they inherit some curvature properties from \cite{W1964}.  To 
describe them we use the obvious decomposition
\begin{equation}\label{z-v}
\gn = \gz + \gv \text{ where $\gz$ is the center and $\Ad(H)\gv = \gv$.}
\end{equation}
Corresponding to (\ref{z-v}), we will use the following notation on
the dual spaces:
\begin{equation}\label{dual-z-v}
\gz^* \ni \zeta \leadsto \lambda_\zeta \in \gn^* \text{ by }
	\lambda_\zeta|_\gz = \zeta \text{ and } \lambda_\zeta|_\gv = 0,
	\text{\quad and \quad} \gn^* \ni \lambda \leadsto 
		\zeta_\lambda = \lambda|_\gz \in \gz^*.
\end{equation}	
The basic conditions with which we'll deal are
\begin{equation}\label{basic}
\begin{aligned}
& N \text{ has square integrable representations modulo its center, }\\
& \gn \text{ has an $\Ad(H)$--invariant symmetric bilinear form $b$
	for which $\gz \perp \gv$,} \\
& \text{the symmetric bilinear form $b$ has nondegenerate restriction 
	$b|_{\gv}$ to $\gv$, and} \\ 	
& \gv \text{ has a complex vector space structure $J$ with $b(Ju,Jv) = b(u,v)$
	for $u, v \in \gz$.}
\end{aligned}
\end{equation}
This will allow us to carry out the program
\begin{equation}\label{program}
\begin{aligned}
 &\text{(a) define a pseudo--K\"ahler structure on $D = \exp(\gv) = N/Z$}, \\
 &\text{(b) construct holomorphic line bundles $\E_\zeta \to D = N/Z$
	for almost every $\zeta \in \gz^*$},\\
&\text{(c) describe the corresponding representations 
	$\pi_{\lambda_\zeta} \in \widehat{N}$
	both on $L^2$ Dolbeault cohomology}\\
 &\text{\phantom{XXX} and on spaces of square integrable
	harmonic $\E_\zeta$--valued differential forms on $V$},\\
&\text{(d) use the underlying holomorphic structure to extend 
	$\pi_{\lambda_\zeta}$ to a linear (not projective)}\\
 &\text{\phantom{XXX} representation of the $H$--stabilizer of 
	$\zeta$, and }\\
 &\text{(e) use this explicit information for the Plancherel and Fourier 
	Inversion formulae for $G$}.
\end{aligned}
\end{equation}

In Section \ref{sec2} we review the algebraic and analytic structure of 
the nilpotent Lie groups $N$ that have irreducible square integrable unitary 
representations.  Most weakly symmetric pseudo--riemannian manifolds can
be viewed as group manifolds of that sort.  These square integrable
representations are the nilpotent group analogs of discrete series
representations of semisimple Lie groups.  We discuss their structure along
the lines of \cite{MW1973} and \cite{W2007}.  Those representations
are basic to our geometric considerations.

In Section \ref{sec3} we look at the domains $D = N/Z$ 
that satisfy (\ref{basic}).   We'll view them as nilpotent group
analogs of flag domains for semisimple (\cite{W1969}, \cite{FHW2005}) 
Lie groups.  We consider the circumstances under which we have invariant
almost complex structures and pseudo--K\"ahler structures on $N/Z$.
Those almost complex structures have constant coefficients in the coordinates
of $\gn/\gz$, so obviously they are integrable, and the
pseudo--k\" ahler structure comes out of geometric quantization theory.
The main point here is the construction of square integrable
Dolbeault cohomology spaces for homogeneous holomorphic vector bundles
over the domains $D$.  We follow the flag domain idea for $N/Z$ and associate
a $N$--homogeneous hermitian holomorphic line bundle $\E_\zeta$ to each
``nonsingular'' $\zeta \in \gz^*$.  We realize the associated
representation $\pi_{\lambda_\zeta}$ as the natural action of $N$ on a
square integrable Dolbeault cohomology space $H_2^{0,\ell}(D;\E_\zeta)$
where $\ell$ is the number of negative eigenvalues of a certain
hermitian form defined by $\zeta$.

In Section \ref{sec4} we extend the constructions of Section \ref{sec3} to
semidirect product groups $G = N\rtimes H$.  The model (which we discuss
later) is the case where $G/H$ is a weakly symmetric pseudo--riemannian
manifold of complex type.  We are especially interested in case of the 
the group $H$ of all automorphisms of $N$ that preserve our 
pseudo--K\"ahler structure on the domain $D = N/Z = G/HZ$, where
the $\E_\zeta \to D$ are $G$--homogeneous.  Those cases occur quite often
in the setting of weakly symmetric pseudo--riemannian manifolds.  All
the ingredients, in the geometric construction of $\pi_{\lambda_\zeta}$ and
$H_2^{0,\ell}(D;\E_\zeta)$, are invariant under the $H$--stabilizer 
$H_\zeta$ of $\zeta$, so $\pi_{\lambda_\zeta}$ extends naturally from
$N$ to a representation $\pi'_{\lambda_\zeta}$ of $N \rtimes H_\zeta$.  
A key point
here is that the geometry lets us bypass the problem of the Mackey 
obstruction.  Then of course we have the induced representations 
$\pi_{\tau,\zeta} := 
\Ind_{NH_\zeta}^G(\tau\widehat{\otimes}\pi'_{\lambda_\zeta})$, 
$\tau \in \widehat{H_\zeta}$\,. 
Since the $\pi_{\lambda_\zeta}$ support the Plancherel measure of $N$, 
the Mackey little--group method shows that the $\pi_{\tau,\zeta}$ 
support the Plancherel measure of $G$.

In Section \ref{sec4} we also indicate the realization of the 
$\pi_{\tau,\zeta}$ both on square integrable partially holomorphic 
cohomology spaces and on  spaces of square integrable partially harmonic 
bundle--valued spinors.

The geometric construction (\ref{rep-g}) of the $\pi_{\tau,\zeta}$ is 
parallel to that of the standard tempered representations of real reductive Lie
groups (\cite{W1973}, or see \cite{W2018}).  The domain $D$
corresponds to a flag domain, $\pi_{\lambda_\zeta}$ corresponds to a relative
discrete series representation of the Levi component of a parabolic
subgroup, and the construction $\pi_{\tau,\zeta} 
= \Ind_{G_\zeta}^G(\tau \widehat{\otimes} \pi'_{\lambda_\zeta})$
corresponds to $L^2$ parabolic induction.  In both settings one can use
partially harmonic square integrable bundle--valued forms as in 
\cite{W1974}, instead of square integrable Dolbeault cohomology.

Finally, in Sections \ref{sec5} and \ref{sec6}, we extract examples from 
the theory of weakly symmetric pseudo--riemannian nilmanifolds, listing the 
holomorphic cases from \cite{WC2018} and recording the signatures
of invariant pseudo--K\"ahlerian metrics.   There, as in the semisimple
setting, the Dolbeault cohomology degree is the number of negative
eigenvalues of the invariant pseudo--K\"ahlerian metric.

In Section \ref{sec5} we review the notion of real form 
family $\{\{G_r/H_r\}\}$ of pseudo--riemannian weakly symmetric nilmanifolds
associated to a riemannian weakly symmetric nilmanifold $G_r/H_r$\,.  
That is considerably more delicate than the semisimple case \cite{CW2017}.
We introduce the notion of ``complex type'' for pseudo--riemannian weakly 
symmetric nilmanifolds.
The pseudo--riemannian weakly symmetric nilmanifolds of complex type 
satisfy the Satake conditions (\ref{A1}) and (\ref{A2}), so the program
(\ref{program}) goes through for them.  Table \ref{heis-comm} lists the
pseudo--riemannian weakly symmetric nilmanifolds of complex type for which
$N$ is a Heisenberg group.  It also shows that we need a maximality
condition in order to have a usable listing for more general $N$, and
Table \ref{max-irred} lists the maximal pseudo--riemannian weakly
symmetric nilmanifolds of complex type for which $H$ acts irreducibly
on $\gn/\gz$.

In Section \ref{sec6} we consider the complete classification of maximal 
pseudo--riemannian weakly symmetric nilmanifolds of complex type.  That 
is necessarily combinatorial and based on a further listing of indecomposable 
maximal pseudo--riemannian weakly symmetric spaces for which $H$ acts
reducibly on $\gn/\gz$ and satisfies some technical conditions.  That
is carried out in Table \ref{indecomp}.

In order to reduce clutter in the notation we will denote
\begin{equation}\label{clutter}
	\pi_\zeta := \pi_{\lambda_\zeta} \text{ for all } \zeta \in \gz^*.
\end{equation}
This notation will be justified by Theorems \ref{mw-nilp} and 
\ref{planch-conc} below.

\section{Square Integrable Representations} \label{sec2}
\setcounter{equation}{0}

In this Section we collect some information on square integrable 
representations for nilpotent Lie groups.  The references are \cite{MW1973} and
\cite{W2007}, with a summary in \cite[Section 2]{W2016}.  These are the
representations and nilpotent groups to which our results apply.  The groups
considered in this note are all of Type I with a countable basis for 
open sets, so there are no measure--theoretic complications.

First, if $B$ is a unimodular Lie group with center $Z$ and 
$\pi \in \widehat{B}$ we have the {\bf central character} 
$\chi_\pi \in \widehat{Z}$ defined by $\pi(z) = \chi_\pi(x)\cdot 1$
for $z \in Z$.  Given $u$ and $v$ in the representation space $\cH_\pi$
we have the {\bf matrix coefficient} or {\bf coefficient function} 
$f_{u,v}: x \mapsto \langle u, \pi(x)v \rangle$,
and $|f_{u,v}|$ is a well defined function on $B/Z$.  Fix Haar measures 
$\mu_B$ on $B$, $\mu_Z$ on $Z$ and $\mu_{B/Z}$ on $B/Z$ such that 
$d\mu_B = d\mu_Z \, d\mu_{B/Z}$\,.  Then these conditions are equivalent:
\begin{equation}\label{sq}
\begin{aligned}
&\text{{\rm (1)} There exist nonzero $u,v \in \cH_\pi$ with 
	$|f_{u,v}| \in L^2(B/Z)$.}\\
&\text{{\rm (2)} $|f_{u,v}| \in L^2(B/Z)$ for all $u,v \in \cH_\pi$.}\\
&\text{{\rm (3)}
$\pi$ is a discrete summand of the representation $\Ind_Z^B(\chi_\pi)$.}
\end{aligned}
\end{equation}

\noindent When those conditions are satisfied for
$\pi \in \widehat{B}$ we say that $\pi$ is {\bf square integrable}
(modulo $Z$).  Then there is a number $\deg \pi > 0$, called the
{\bf formal degree} of $\pi$,
such that
\begin{equation}
\int_{G/Z} f_{u,v}(x)\overline{f_{u',v'}(x)}d\mu_{G/Z}(xZ)
        = \tfrac{1}{\deg \pi} \langle u,u' \rangle
                \overline{\langle v,v'\rangle}
\end{equation}
for all $u,u',v,v' \in \cH_\pi$\,.
If $\pi_1, \pi_2 \in \widehat{G}$ are inequivalent and satisfy (\ref{sq}),
and if $\chi_{\pi_1} = \chi_{\pi_2}$, then
\begin{equation}
\int_{G/Z} \langle u, \pi_1(x)v \rangle 
        \overline{\langle u', \pi_2(x)v' \rangle}d\mu_{G/Z}(xZ) = 0
\end{equation}
for all $u,v \in \cH_{\pi_1}$ and all $u', v' \in \cH_{\pi_2}$\,.

The main results of \cite{MW1973} shows exactly how this works for
nilpotent Lie groups:

\begin{theorem}  \label{mw-nilp}
Let $N$ be a connected simply connected nilpotent Lie group with center $Z$,
$\gn$ and $\gz$ their Lie algebras, and $\gn^*$ the linear dual space
of $\gn$.  Let $\lambda \in \gn^*$ and let $\pi_\lambda$ denote the
irreducible unitary representation attached to the coadjoint orbit 
$\Ad^*(N)\lambda$ by
the Kirillov theory {\rm \cite{K1962}}.  Then the following conditions are
equivalent.

{\rm (1)} $\pi_\lambda$ satisfies the conditions of {\rm (\ref{sq})}.

{\rm (2)} The coadjoint orbit $\Ad^*(N)\lambda = \{\nu \in \gn^* \mid 
        \nu|_\gz = \lambda|_\gz\}$.

{\rm (3)} The bilinear form $b_\lambda(x,y) = \lambda([x,y])$ on $\gn/\gz$
        is nondegenerate.

\noindent
The Pfaffian $\Pf(b_\lambda)$ is a polynomial function
$P(\lambda|_{\gz})$ on $\gz^*$.  The set of equivalence classes, of 
representations
$\pi_\lambda$ for which these conditions hold, is parameterized by the set
$\{\zeta \in \gz^* \mid P(\zeta) \ne 0\}$ $($which is empty or Zariski open
in $\gz^*)$.
\end{theorem}

We will say that the connected simply connected nilpotent Lie group $N$ is
{\bf square integrable} if it has a square integrable irreducible unitary
representation, in other words if there exists $\lambda \in \gn^*$ such that
$P(\lambda|_{\gz}) \ne 0$.  We remark that the  group $N$ is 
square integrable if and 
only if the universal enveloping algebra $\cU(\gz)$ is the center of $\cU(\gn)$
\cite{MW1973}.  

Recall that if $\zeta \in \gz^*$ then (\ref{clutter}) 
$\pi_\zeta$ denotes the $\pi_\lambda$ for which $\lambda|_\gv = 0$ and
$\lambda|_\gz = \zeta$.

\begin{theorem} \label{planch-conc}
Let $N$ be a square integrable connected simply connected nilpotent Lie group
with center $Z$.  Then Plancherel measure on $\widehat{N}$ is
concentrated on $\{\pi_\zeta \mid P(\zeta) \ne 0\}$,  
where it is a positive multiple of the absolutely continuous measure 
$|P(\zeta)|d\zeta$.
The formal degree $\deg \pi_\zeta = |P(\zeta)|$.
\end{theorem}

Given $\zeta \in \gz^*$ with $P(\zeta) \ne 0$ and a Schwartz class
function $f \in \cS(N)$, $\cO(\zeta)$ denotes the co-adjoint orbit
$\Ad^*(N)(\lambda_\zeta) = \zeta + \gz^\perp$\,, 
$f_\zeta = (f\cdot \exp)|_{\cO(\zeta)}$ 
and $\widehat{f_\zeta}$ is the
Fourier transform of $f_\zeta$ on $\cO(\zeta)$.

\begin{theorem} \label{sq-inv}
Let $N$ be a square integrable connected simply connected nilpotent Lie group
with center $Z$.  Let $f \in \cS(N)$.
If $\zeta \in \gz^*$ with $P(\zeta) \ne 0$ then the distribution
character of $\pi_\zeta$ is given by
\begin{equation}\label{char-sq}
\Theta_{\pi_\zeta}(f) = 
        \tr \int_N f(x) \pi_\zeta(x) d\mu_G(x) =
        c^{-1}|P(\zeta)|^{-1}\int_{\nu \in \cO(\zeta)} \widehat{f_\zeta} \,d\nu
\end{equation}
where $c = d!2^d \text{ and } d = \dim(\gn / \gz)/2$\,, and $d\nu$ is
ordinary Lebesgue measure on the affine space $\cO(\zeta)$\,.
The Fourier Inversion formula for $N$ is
\begin{equation}\label{fi-sq}
	f(x) = c\int_{\gz^*} \Theta_{\pi_\zeta}(r_xf)
                |P(\zeta)|\, d\zeta \text{ where }
        (r_xf)(y) = f(yx) \text{ (right translate)}.
\end{equation}
\end{theorem}

\section{Holomorphic Line Bundles over Domains $N/Z$}\label{sec3}
\setcounter{equation}{0}

Let $N$ be a connected, simply connected, nilpotent Lie group.  Let
$Z$ denote the center of $N$.  Their Lie algebras satisfy 
$\gn = \gz + \gv$ where $\gv$ is a vector space complement to $\gz$
in $\gn$.  Let $H$ be a reductive group of automorphisms on $N$.  
Consider the semidirect product
\begin{equation}\label{H}
G = N \rtimes H, \text{ so } \gg = \gz + \gv + \gh.
\end{equation}
Then automatically $\Ad(H)\gz = \gz$.  We make our choice of $\gv$ 
so that $\Ad(H)\gv = \gv$.
Later we will impose further conditions on these Lie algebras.  For
the moment we only require Satake's conditions.  

First, we assume that $\gv$ has an $\Ad(H)$--invariant complex structure $J$
whose $(\pm i)$ eigenspaces $\gv_\pm$ satisfy
\begin{equation}\label{A1}
[\gv_+,\gv_+] \subset \gv_+ + \gz_\C \text{ and } 
	[\gv_-,\gv_-] \subset \gv_- + \gz_\C
\end{equation}
Then $[\gh_\C, \gv_\pm] \subset \gv_\pm$\,, so each $\gh_\C + \gv_\pm + \gz_\C$
is a subalgebra of $\gg_\C$, very much like a parabolic,
with nilradical $\gv_\pm + \gz_\C$ and Levi component $\gh_\C$\,.

On the group level we suppose that $G$ is contained in its complexification
$G_\C$, so $N_\C = Z_\C V_\C$ where $V = \exp(\gv)$, and 
$G_\C = N_\C \rtimes H_\C$\,.  Let $V_\pm = \exp(\gv_\pm)$, so each $Z_\C V_\pm$
is a closed complex analytic subgroup of $N_\C$\,.  For convenience we denote
\begin{equation}\label{npm}
N_\pm = Z_\C V_\pm \text{ and } \gn_\pm = \gz_\C + \gv_\pm\,.
\end{equation}
\begin{lemma}\label{A2}
$D = G/HZ = G_\C/H_\C N_- \cong N_\C/ N_-$\,.
\end{lemma}
\begin{proof}
From (\ref{npm}) we have $G\cap H_\C N_\pm = HZ$ and $H_\C \cap N_\pm = \{1\}$.
By dimension, $GH_\C N_\pm$ is open in $G_\C$
and $G/HZ \simeq GH_\C N_\pm/H_\C N_\pm$ is open in $G_\C/H_\C N_\pm$\,.
Thus $D = G/HZ$ is an open $G$--orbit in the complex homogeneous
space $G_\C/H_\C N_- \cong N_\C/ N_-$\,.  But $G/HZ$ is an $N$--orbit, and
orbits of unipotent groups on affine manifolds are closed.  Thus
$D = G/HZ$ is a closed $G$--orbit $G_\C/H_\C N_- \cong N_\C/ N_-$\,.
The Lemma follows.
\end{proof}
We will work with $D$ in a way suggested by realization of discrete series 
of semisimple Lie groups representations over flag domains (\cite{NO1970},
\cite{Sc1971, S1976}, \cite{W1973, W1974}).  

\begin{lemma}\label{L0}
Let $\zeta \in \gz^*$.  Consider the symmetric bilinear form
$\beta_\zeta$ on $\gv$ given by $\beta_\zeta(u,v) = \lambda_\zeta([u,Jv])$.
Then $\beta_\zeta$ is nondegenerate if and only if $\lambda_\zeta$
satisfies the conditions of {\rm Theorem \ref{mw-nilp}}.
\end{lemma}
\begin{proof} In the notation of Theorem \ref{mw-nilp}, 
$\beta_\zeta$ is nondegenerate if and only if $b_{\lambda_\zeta}$ is
nondegenerate.
\end{proof}

Now let $\zeta \in \gz^*$ with $\beta_\zeta$ nondegenerate.   Then
\begin{equation}\label{L3}
	\chi_\zeta := \exp(i \lambda_\zeta|_{N_-})) 
	\text{ is a holomorphic character on } N_-\,.
\end{equation}
Using Lemma \ref{A2}, $\chi_\zeta$ defines 
\begin{equation}\label{L4}
\begin{aligned}
&\E_\zeta \to D = G/HZ = G_\C/H_\C N_- \cong N_\C/ N_-\,: \\
&\text{\quad $G_\C$--homogeneous holomorphic line bundle 
associated to } \chi_\zeta\,.
\end{aligned}
\end{equation}
Then we have
\begin{equation}\label{L5}
C_c^{p,q}(D;\E_\zeta): \text{ compactly supported $\E_\zeta$--valued 
	$(p,q)$ forms on } D
\end{equation}
Choose an $N$--invariant positive definite hermitian inner product $\gamma$ on
(the fibers of ) $\E_\zeta$\,.  Then we have the usual Hodge--Kodaira
orthocomplementation operator $\sharp$ sending $\E_\zeta$--valued
$(p,q)$ forms to $\E^*_\zeta$--valued $(n-p,n-q)$--forms, $n = \dim_\C D$,
and the formal adjoint $\overline{\partial}^* = -\sharp\partial\sharp$ of
$\overline{\partial}: C_c^{p,q}(D;\E_\zeta) \to C_c^{p,q+1}(D;\E_\zeta)$.
That gives us the Hodge--Kodaira--Laplace operator
\begin{equation}\label{L6}
\square = \overline{\partial}\, \overline{\partial}^* + 
	\overline{\partial}^* \overline{\partial}.
\end{equation}
The closure and the adjoint of $\square$ are equal, giving a self--adjoint
extension (which is also denoted $\square$) to 
\begin{equation}\label{L7}
L_2^{p,q}(D;\E_\zeta): \text{ square integrable $\E_\zeta$--valued
	$(p,q)$ forms on } D.
\end{equation}
That defines the space of square integrable harmonic 
$\E_\zeta$--valued $(p,q)$ 
forms on $D$:
\begin{equation}\label{L8}
H_2^{p,q}(D;\E_\zeta) = \{\omega \in L_2^{p,q}(D;\E_\zeta)
	\mid \square(\omega) = 0\}.
\end{equation}
By elliptic regularity of $\square$, $H_2^{p,q}(D;\E_\zeta)$ consists of
$C^\infty$ forms, in fact $C^\infty$ Schwartz class forms,
There the domain of $\square$ is all of $\cS^{p,q}(D;\E_\zeta)$, and
$H_2^{p,q}(D;\E_\zeta) \subset \cS^{p,q}(D;\E_\zeta)$.
As defined, $H_2^{p,q}(D;\E_\zeta)$ depends on the choice of 
positive definite hermitian inner product $\gamma$, but we can avoid that
issue using the orthogonal decomposition of Andreotti and Vesentini
\cite{AV1965}:
\begin{equation}\label{AV}
L_2^{p,q}(D;\E_\zeta) 
	= c\ell\,\,\overline{\partial}^* L_2^{p,q+1}(D;\E_\zeta)
	+ c\ell\,\, \overline{\partial}L_2^{p,q-1}(D;\E_\zeta) 
	+ H_2^{p,q}(D;\E_\zeta)
\end{equation}
where $c\ell$ denotes $L^2$ closure.
Thus we may (and do) identify $H_2^{p,q}(D;\E_\zeta)$ as a Hilbert space
completion of square integrable Dolbeault cohomology based on smooth Schwartz
class forms,
$$ 
H_2^{p,q}(D;\E_\zeta) \cong 
  {\rm Kernel\,} \bigl ( \overline{\partial}: \cS^{p,q}(D;\E_\zeta)
        \to \cS^{p,q+1}(D;\E_\zeta)\bigr ) /
  {\rm Image\,} \bigl (\overline{\partial}: \cS^{p,q-1}(D;\E_\zeta)
        \to \cS^{p,q}(D;\E_\zeta)\bigr ).
$$
The theorem of Satake \cite[Proposition 1]{Sa1971}and Okamoto (unpublished),
in this setting, can be reformulated as follows.  Consider the
hermitian form 
\begin{equation}\label{H1}
\gamma_\zeta(u,v) = \lambda_\zeta([u,Jv]) + i\, \lambda_\zeta([u,v]) 
	\text{ where } u,v \in \gv.
\end{equation}
Here $\lambda_\zeta([u,Jv]) = \beta_\zeta(u,v)$ is real symmetric on $\gv$ 
of some signature $(2k,2\ell)$ and $\lambda_\zeta([u,v])$ is antisymmetric,
so $\gamma_\zeta(u,v)$ is (complex) hermitian on $(\gv,J)$ of
corresponding signature $(k,\ell)$.  Thus $k$ is the dimension of any 
maximal positive definite subspace of $(\gv,J)$ and $\ell$ is the 
dimension of any maximal negative definite subspace.  Note that 
$\beta_\zeta$ is nondegenerate if and only if $P(\zeta) \ne 0$.
\begin{proposition}\label{SaO}
Let $\zeta \in \gz^*$ such that the hermitian form
$\gamma_\zeta(u,v)$ on $\gv$ is nondegenerate.
Let $n = \dim_\C(\gv,J)$ and let $(k,\ell)$ for the
signature of $\gamma_\zeta$ on $(\gv,J)$.   Then
$H^{0,q}_2(D;\E_\zeta) = 0$ for $q \ne \ell$ and the natural 
action of $N$ on $H^{0,\ell}_2(D;\E_\zeta)$ is the irreducible
unitary representation $\pi_\zeta$ with central character $\chi_\zeta$\,.
\end{proposition}

In view of Lemma \ref{L0} and the decomposition (\ref{AV}),
Proposition \ref{SaO} shows that the square integrable cohomology 
representation of $N$ corresponding to $\zeta$ is independent (up to 
unitary equivalence) of the
choice of $J$.  The cohomology degree, however, will
depend on choice of $J$, as seen in \cite{W1975}.  Further, since
$N$ satisfies the conditions of Theorem \ref{mw-nilp}, the Plancherel
measure and the Plancherel Formula and Fourier Inversion theorems
for $N$ are given by Theorems \ref{planch-conc} and \ref{sq-inv}.

\section{Extension to the Semidirect Product Group}\label{sec4}
\setcounter{equation}{0}
In this section we extend the results of Section \ref{sec3} from
the nilpotent group $N$ to the semidirect product group $G = N \rtimes H$.

\begin{lemma}\label{stab0}
Let $\zeta \in \gz^*$ with $P(\zeta) \ne 0$.  
Define $H_\zeta = \{h \in H \mid \Ad^*(h)\zeta = \zeta\}$ and
$G_\zeta = NH_\zeta$\,.  Then $G_\zeta$ is the subgroup of $G$
for which $\pi_\zeta\cdot \Ad(g)$ is equivalent to $\pi_\zeta$.
In other words, $G_\zeta$ is the Mackey little--group in $G$ for 
$\pi_\zeta$\,.
\end{lemma}
\begin{proof}
Let $h \in H$.  If $\Ad^*(h)\zeta = \zeta$ then 
$\pi_{\lambda_\zeta}\cdot \Ad(h)$ is equivalent to $\pi_{\lambda_\zeta}$ 
by Kirillov theory, in other words $\pi_\zeta\cdot\Ad(h)$ is equivalent
to $\pi_\zeta$.  If $\pi_\zeta\cdot \Ad(h)$ is equivalent to $\pi_\zeta$ then
$\Ad^*(N)(\lambda_\zeta\cdot \Ad(h)) = \Ad^*(N)(\lambda_\zeta)$.  As $H$ is 
reductive $\Ad^*(h)$ preserves $\gz^*$ and its complement $\gv^*$, so it
preserves $\Ad^*(N)(\lambda_\zeta) \cap \gz^* = \{\zeta\}$.  Now $H_\zeta$
is the $H$--stabilizer of $\pi_\zeta$, so $G_\zeta = NH_\zeta$
is the $G$--stabilizer.
\end{proof}

Since $H_\zeta$ preserves every ingredient in the construction
of $\pi_\zeta$ given by Proposition \ref{SaO} we now have
\begin{proposition}\label{stab1}
$\pi_\zeta$ extends to a
unitary representation $\pi'_\zeta$ of $G_\zeta$ on 
the representation space $\cH_\zeta$ of $\pi_\zeta$\,.
\end{proposition}

In view of Theorem \ref{planch-conc}, the Mackey little--group method
gives us
\begin{proposition}\label{supp-planch-g}
Plancherel measure on $\widehat{G}$ is concentrated on the representations
$$
\pi_{\tau,\zeta} = 
     \Ind_{G_\zeta}^G(\tau \widehat{\otimes} \pi'_\zeta)
\text{ where } \zeta \in \gz^* \text{ with } P(\zeta) \ne 0
\text{ and where } \tau \in \widehat{H_\zeta}\,.
$$
\end{proposition}

Now we extend Proposition \ref{SaO} from $N$ to $G$\,.
Let $\zeta \in \gz^*$ with $P(\zeta) \ne 0$.  Let
$\tau \in \widehat{H_\zeta}$ with representation space $E_\tau$\,,
and let $\E_\tau \to D$ denote the corresponding $G_\zeta$--homogeneous
holomorphic vector bundle.  Similarly let $E_\zeta$ denote the
complex line that is the representation space of $\chi_\zeta$; it
led to our $G_\zeta$--homogeneous holomorphic line bundle $\E_\zeta \to D$. 
Recall the notation of Proposition \ref{SaO}.  Then 
$E_\tau \widehat{\otimes} H^{0,\ell}_2(D;\E_\zeta)$ is the representation
space of $\tau \widehat{\otimes} \pi'_\zeta$\,.  Denote
\begin{equation}\label{geom-setup}
(\H_{\tau,\zeta} = \H_\tau \otimes \E_\zeta) \to (D=G_\zeta/H_\zeta Z):
\text{ associated vector bundle with fiber } E_\tau \otimes E_\zeta\,.
\end{equation}

Express $D = G_\zeta/H_\zeta Z$\,.
The isotropy $H_\zeta Z$ preserves the infinitesimal right action 
of the antiholomorphic tangent space $\gv_-$ of $D$, so $\gv_-$ acts on
the right on smooth local sections of $\H_{\tau,\zeta} \to D$\,.  In
other words we have a well defined $\overline{\partial}$--operator on
smooth local sections of $\H_{\tau,\zeta} \to D$.  Thus
\begin{lemma}\label{stab2}
$\H_{\tau,\zeta}\to D$ is a hermitian $G_\zeta$--homogeneous holomorphic
vector bundle with $\overline{\partial}$--operator given by the right
action of $\gv_-$\,.
\end{lemma}

Now we have the Hodge--Kodaira--Laplace operator $\square$ as in (\ref{L6}). 
As in that case, where $\tau$ is the trivial representation,
$\square$ acts on the dense subspace $C_c^{p,q}(D;\H_{\tau,\zeta})$ of 
$L_2^{p,q}(D;\H_{\tau,\zeta})$--valued smooth $(p,q)$--forms on $D$.
Note that its action only affects the $\E_\zeta$ component of the values
of local sections.  Thus, as before, the closure and adjoint of $\square$
are equal, so $\square$ is essentially self adjoint, and we have its kernel
\begin{equation}\label{L9}
H_2^{p,q}(D;\H_{\tau,\zeta}) = \{\omega \in L_2^{p,q}(D;\H_{\tau,\zeta})
        \mid \square(\omega) = 0\},
\end{equation}
the space of $\H_{\tau,\zeta}$--valued square integrable harmonic
$(p,q)$--forms on $D$.  Applying Proposition \ref{SaO} we have
\begin{proposition}\label{SaO-lambda}
Let $\zeta \in \gz^*$ such that $\gamma_\zeta$ {\rm (from \ref{H1})}
is nondegenerate with signature $(k,\ell)$ on $\gv$.  Then 
$H_2^{0,q}(D;\H_{\tau,\zeta}) = 0$ for $q \ne \ell$, and the natural
action of $G_\zeta = H_\zeta N$ on $H_2^{0,\ell}(D;\H_{\tau,\zeta})$
is the unitary representation $\tau \widehat{\otimes}\pi'_\zeta$\,.
\end{proposition}
Again we can apply \cite{AV1965} to see
\begin{equation}\label{AV1}
L_2^{p,q}(D;\H_{\tau,\zeta}) 
        = c\ell\,\,\overline{\partial}^* L_2^{p,q+1}(D;\H_{\tau,\zeta})
        + c\ell\,\, \overline{\partial}L_2^{p,q-1}(D;\H_{\tau,\zeta}) 
        + H_2^{p,q}(D;\H_{\tau,\zeta})
\end{equation}
where $c\ell$ denotes $L^2$ closure.
Making use of elliptic regularity of $\square$
we identify $H_2^{p,q}(D;\H_{\tau,\zeta})$ as a Hilbert space
of square integrable Dolbeault cohomology based on 
Schwartz class forms,
$$ 
H_2^{p,q}(D;\H_{\tau,\zeta}) \cong 
  {\rm Kernel\,} \bigl ( \overline{\partial}: \cS^{p,q}(D;\H_{\tau,\zeta})
	\to \cS^{p,q+1}(D;\H_{\tau,\lambda})\bigr ) /
  {\rm Image\,} \bigl (\overline{\partial}: \cS^{p,q-1}(D;\H_{\tau,\zeta})
	\to \cS^{p,q}(D;\H_{\tau,\zeta})\bigr ).
$$
Thus $H_2^{p,q}(D;\H_{\tau,\zeta})$ is a complete locally convex topological 
vector space and is independent of choice of the hermitian inner product 
on $\E_\zeta$ used to define $\square$ on $\H_{\tau,\lambda}$\,.

Now let us return to the representations 
$\pi_{\tau,\zeta} = \Ind_{G_\zeta}^G(\tau \widehat{\otimes} \pi'_\zeta) 
\in \widehat{G}$ of Proposition \ref{supp-planch-g}.  The representation
space of $\pi_{\tau,\zeta}$ is
\begin{equation}\label{rep-g}
\begin{aligned}
&\cH_{\pi_{\tau,\zeta}} = \{f: G \to H_2^{0,\ell}(D;\H_{\tau,\zeta}) \mid
	f(gx) = (\tau \widehat{\otimes} \pi'_\zeta)(x)^{-1}f(g)
	\text{ for } x \in G_\zeta \\
&\text{with inner product given by } ||f||^2 =
	\int_{G/G_\zeta} ||f(gG_\zeta)||^2 d(gG_\zeta) < \infty.
\end{aligned}
\end{equation}
The extension of Theorem \ref{planch-conc} to $G$ is
\begin{theorem} \label{planch-conc-g}
Let $N$ be a square integrable connected simply connected nilpotent Lie group
with center $Z$.  Let $H$ be a reductive group of automorphisms of $N$
that preserves a nondegenerate symmetric bilinear form on $\gn/\gz$.
Let $G = N \rtimes H$ and suppose that {\rm (\ref{A1})} and 
{\rm (\ref{A2})} hold.  Then Plancherel measure on $\widehat{G}$ is
concentrated on $\{\pi_{\tau,\zeta} \mid \zeta \in \gz^*, P(\zeta) \ne 0,
\text{ and } \tau \in \widehat{H_\zeta}\}$.
\end{theorem}

The construction (\ref{rep-g}) of the $\pi_{\tau,\zeta}$ is analogous
to that of the standard tempered representations of real reductive Lie
groups (\cite{W1973}, \cite{W2018}).  For that, the domain $D$ 
corresponds to a flag domain, $\pi_\zeta$ corresponds to a relative
discrete series representation of the Levi component of a parabolic
subgroup, and the construction
$\pi_{\tau,\zeta} = \Ind_{G_\zeta}^G(\tau \widehat{\otimes} \pi'_\zeta)$
corresponds to $L^2$ parabolic induction.  In both
settings, the geometric realizations can occur both on spaces of 
partially harmonic square integrable bundle--valued spinors and
on square integrable partially holomorphic cohomology spaces.

\section{Weakly Symmetric Pseudo--Riemannian Nilmanifolds}\label{sec5}
\setcounter{equation}{0}

The theory of weakly symmetric pseudo--riemannian nilmanifolds provides 
many interesting examples of the spaces $G/H$, $G = N\rtimes H$, studied
in Sections \ref{sec3} and \ref{sec4} above.  We list the more accessible
examples in Sections \ref{sec5} and \ref{sec6}.  Here we start by
sketching some elements of the theory.  In the next section we also
discuss the classification.

Recall that a riemannian manifold $(S,ds^2)$ is {\bf symmetric} if, given
$x \in S$, there is an isometry $s_x$ of $(S,ds^2)$ such that $s_x(x) = x$
and $ds_x(\xi) = -\xi$ for every tangent vector $\xi \in T_x(S)$.  
It is {\bf weakly symmetric} if, given $x \in S$ and $\xi \in T_x(S)$,
there is an isometry $s_{x,\xi}$ of $(S,ds^2)$ such that $s_{x,\xi}(x) = x$
and $ds_{x,\xi}(\xi) = -\xi$.  The obvious difference is that $s_{x,\xi}$
depends on $\xi$ as well as $x$.  Many properties of symmetric spaces hold 
in the weakly symmetric setting, for example homogeneity, the geodesic 
orbit property, commutativity of the $L^1$
convolution algebra, the theory of spherical functions, and Plancherel and
Fourier inversion formulae; see \cite{W2007} for an exposition.  But the
associated Lie group theory and the classification theory are quite different.  

There are many weakly symmetric riemannian nilmanifolds, i.e. 
weakly symmetric riemannian manifolds 
that admit a transitive nilpotent group of isometries.  By contrast
the only symmetric riemannian nilmanifolds are the flat ones; they are the
products of flat tori and euclidean spaces.  Here is the best known
example of this phenomenon.  Let $\gn$ be the Heisenberg Lie algebra
of real dimension $2n+1$,
$\gn = \Im \C + \C^n$ with composition $[(z,v),(z',v')] = \Im \langle v,
v' \rangle$ where $\langle \cdot , \cdot \rangle$ is the usual positive
definite hermitean inner product on $\C^n$.  The unitary group $U(n)$
acts by isomorphisms, $h: (z,v) \mapsto (z,hv)$.  That gives us the
semidirect product group $G = N \rtimes U(n)$.  That results in 
weakly symmetric $G$--invariant riemannian metrics on $N = G/U(n)$.
None of the corresponding weakly symmetric spaces are symmetric.
See \cite{W2007} for an exposition of Yakimova's classification
(\cite{Y2004}, \cite{Y2005}, \cite{Y2006}) of weakly symmetric 
riemannian manifolds.

The theory of weakly symmetric pseudo--riemannian manifolds is much 
more delicate, and in fact there are several competing definitions.
We will consider the most accessible one, that of 
real forms of weakly symmetric riemannian manifolds.

Let $(M_r\,, ds_r^s)$ be a connected weakly symmetric riemannian manifold.
Suppose that $M_r = G_r/H_r$ is a riemannian nilmanifold, in other
words that $G_r = N_r\rtimes H_r$ where $N_r$ is a connected nilpotent
Lie group acting transitively on $M_r$\,.  The associated {\bf real form 
family} $\{\{G_r/H_r\}\}$ consists of all $G/H$ 
with the same complexification $(G_r)_\C/(H_r)_\C$.  In other words 
$H$ is a real form of $(H_r)_\C$\,, and $G = N \rtimes H$ where 
$N$ is an $\Ad(H)$--invariant real form of $(N_r)_\C$\,.  
See \cite{WC2018} for
the definition and a discussion of the $\Ad(H)$--invariance condition.
These $M=G/H$, with invariant pseudo--riemannian metric $ds^2$, are our 
weakly symmetric pseudo--riemannian nilmanifolds.  
Every weakly symmetric riemannian manifold is a commutative space, 
and we work a little bit more generally,
assuming that $M_r$ is a commutative nilmanifold.

\begin{definition}\label{cpx}{\rm
A weakly symmetric pseudo--riemannian nilmanifold
$M = G/H$ is of {\bf complex type} if it satisfies the conditions
(\ref{A1}) and (\ref{A2}).} \hfill $\diamondsuit$
\end{definition}  

\begin{example}\label{w-circle}{\rm
Let $M = G/H$ be a weakly symmetric pseudo--riemannian nilmanifold,
say $G = N\rtimes H$ and $\gn = \gz + \gv$ as in {\rm (\ref{H})}.
Suppose that $\Ad_G(H)$ is irreducible on $\gv$ and that $H$ has
a central subgroup $T \cong U(1)$.  Let $\zeta \in \gt$ such that
$J := \Ad(\exp(\zeta))|_\gv$ has square $-I$.   Then $J$ is an
$\Ad(H)$--invariant complex structure on $\gv$ with which
$M = G/H$ is of complex type.} \hfill $\diamondsuit$
\end{example}

There are many cases, as we see in Table \ref{heis-comm} below,
of weakly symmetric pseudo--riemannian nilmanifolds $M_i = G_i/H_i$
with $G_i = N \rtimes H_i$ and $H_2 \subsetneqq H_1$\,.  There in
both cases we have the same $\gn = \gz + \gv$.
If $M_1 = G_1/H_1$ is of complex type as defined by a central circle
subgroup of $H_1$ as in 
Example \ref{w-circle}, the $\Ad(H_1)$--invariant complex structure
$J$ on $\gv$ is $\Ad(H_2)$--invariant as well, so 
$M_2 = G_2/H_2$ is of complex type, and $\gv$ has the same
signature for both.

We now extract a number of examples of weakly symmetric
pseudo--riemannian nilmanifolds of complex type from \cite{WC2018}, 
to which the results of Section \ref{sec4} apply.  Those are manifolds 
$(M,ds^2)$, $M = G/H$ with $G = N\rtimes H$, $N$ nilpotent and $H$ 
reductive in $G$, such that the pseudo--riemannian metric $ds^2$ is
the real part of an $H$--invariant pseudo--K\"ahler metric.  The first
examples are, of course, those for which $N$ is a Heisenberg group and
$H$ acts $\R$--irreducibly on $\gv$.  Table \ref{heis-comm} just below 
extracts
them from \cite[Table 4.2]{WC2018}.  For the convenience of the reader 
who wants to check this passage to real forms we retain the numbering 
of real form families as in \cite[Table 4.2]{WC2018}.  For the signature 
of $\gv$ we give the
signature $(2k,2\ell)$ of the real symmetric bilinear form $\beta_\zeta$
on $\gv$ for a choice of nonzero $\zeta$; if we used $-\zeta$ instead, 
then the signature on $\gv$ would be the reverse, $(2\ell,2k)$.  Then, of 
course, the signature of the hermitian form $\gamma_\zeta$ on $\gv$ 
is $(k,\ell)$, and of $\gamma_{-\zeta}$ is $(\ell,k)$.  For brevity we only 
list one of $(k,\ell)$ and $(\ell,k)$.

\addtocounter{equation}{1}
\begin{longtable}{|r|l|l|l|}
\caption*{\bf {\normalsize Table} \thetable \quad {\normalsize
        Irreducible Commutative Heisenberg Nilmanifolds
        $(N\rtimes H)/H$}} \label{heis-comm} \\
\hline
 & Group $H$ & $\gv \text{ and signature}(\gv)$&
                $\gz$  \\ \hline
\hline
\hline
\endfirsthead
\multicolumn{4}{l}{{\normalsize \textit{Table \thetable\, continued from
        previous page $ \dots$}}} \\
\hline
 & Group $H$ & $\gv \text{ and signature}(\gv)$&
                $\gz$  \\ \hline
\hline
\endhead
\hline \multicolumn{4}{r}{{\normalsize \textit{$\dots$ Table \thetable\,
        continued on next page}}} \\
\endfoot
\hline
\endlastfoot
\hline
{\rm 1} & $SU(r,s)$  &  $\C^{r,s}, \,\,\, (2r,2s)$ & $\Im\C$
        \\ \hline \hline
{\rm 2}  & $U(r,s)$ &  $\C^{r,s}, \,\,\, (2r,2s)$ & $\Im\C$
        \\ \hline\hline
{\rm 3} & $Sp(k,\ell)$ & $\C^{2k,2\ell},\,\,\, (4k,4\ell)$ &$\Im\C$
        \\ \hline\hline
{\rm 4} & $U(1)\cdot Sp(k,\ell)$ & $\C^{2k,2\ell},\,\,\, (4k,4\ell)$
        & $\Im\C$ \\ \hhline{|~|-|-|-|}
   & $U(1) \cdot Sp(m;\R)$ & $\C^{m,m},\,\,\, (2m,2m)$ & $\Im\C$
        \\ \hline\hline
{\rm 5} & $SO(2)\cdot SO(r,s), r+s\geqq 2$ & $\R^{2\times(r,s)},\,\,\, (2r,2s)$
        &  $\Im\C$ \\ \hhline{|~|-|-|-|}
   & $U(1)\cdot SO^*(n), n$ even & $\C^n \simeq \R^{n,n},\,\,\, (n,n)$ & $\Im\C$
        \\ \hline\hline
{\rm 6} & $U(k,\ell)$ & $S_\C^2(\C^{k,\ell}),\,\,\, (k^2+k+\ell^2+\ell,2k\ell))$
        & $\Im\C$
         \\ \hline\hline
{\rm 7} & $SU(k,\ell), k+\ell$ odd & 
	$\Lambda_\C^2(\C^{k,\ell}),\,\,\,(k^2-k+\ell^2-\ell,2k\ell)$
         & $\Im\C$ \\ \hline\hline
{\rm 8} & $U(k,\ell)$ & $\Lambda^2_\C(\C^{k,\ell}),\,\,\,(k^2-k+\ell^2-\ell,2k\ell))$
	&$\Im\C$ \\ \hline\hline
{\rm 9} & $SU(k,\ell) \cdot SU(r,s)$ & $\C^{(k,\ell) \times (r,s)},\,\,\,
        (2kr+2\ell s, 2ks + 2\ell r)$ & $\Im\C$\\ \hhline{|~|-|-|-|}
   & $SL(\tfrac{m}{2};\H) \cdot SL(\tfrac{n}{2};\H)$ & $\C^{m \times n},\,\,\, 
        (mn,mn)$ & $\Im\C$ \\ \hline\hline
{\rm 10} & $S(U(k,\ell) \cdot U(r,s))$ & $\C^{(k,\ell) \times (r,s)},\,\,\,
        (2kr+2\ell s, 2ks + 2\ell r)$ & $\Im\C$\\ \hhline{|~|-|-|-|}
   & $S(GL(\tfrac{m}{2};\H) \cdot GL(\tfrac{n}{2};\H))$
        & $\C^{m \times n},\,\,\, 
        (mn,mn)$ & $\Im\C$ \\ \hline\hline
{\rm 11} & $U(a,b) \cdot Sp(k,\ell),  a+b=2$
         & $\C^{a,b} \otimes_\C \C^{2k,2\ell},\,\,\,
        (4ak+4b\ell,4a\ell +4bk)$ & $\Im\C$\\ \hhline{|~|-|-|-|}
   & $U(a,b)\cdot Sp(m;\R), a+b=2$ & $\C^{a,b}\otimes_\C \C^{2m},\,\,\,
        (4m,4m)$ & $\Im\C$\\ \hline\hline
{\rm 12} & $SU(a,b) \cdot Sp(k,\ell)$ 
	& $\C^{a,b} \otimes_\C \C^{2k,2\ell},\,\,\,
        (4ak+4b\ell,4a\ell +4bk)$ & $\Im\C$ \\ \hline\hline
{\rm 13} & $U(a,b) \cdot Sp(k,\ell), a+b=3$ 
        & $\C^{a,b} \otimes_\C \C^{2k,2\ell},\,\,\,
        (4ak+4b\ell,4a\ell +4bk)$ & $\Im\C$ \\ \hhline{|~|-|-|-|}
   & $U(a,b) \cdot Sp(m;\R), a+b=3$ & $\C^{a,b} \otimes_\C \C^{2m},\,\,\,
        (6m,6m)$ & $\Im\C$ \\ \hline\hline
{\rm 14} & $U(a,b) \cdot Sp(k,\ell), \begin{smallmatrix} a+b=4\\k+\ell = 4
        \end{smallmatrix}$ & $\C^{a,b} \otimes_\C \C^{2k,2\ell},\,\,\,
        (4ak+4b\ell,4a\ell +4bk)$ & $\Im\C$ \\ \hhline{|~|-|-|-|}
   & $U(a,b) \cdot Sp(4;\R), a+b=4$ & $\C^{a,b} \otimes_\C \C^8,\,\,\,
        (32,32)$ & $\Im\C$ \\ \hline\hline
{\rm 15} & $SU(k,\ell) \cdot Sp(r,s), r+s=4$ 
	& $\C^{k,\ell} \otimes_\C \C^{2r,2s},\,\,\,
        (4kr+4\ell s,4ks +4\ell k)$ & $\Im\C$\\ \hline\hline
{\rm 16} & $U(k,\ell) \cdot Sp(r,s), \begin{smallmatrix} k+\ell \geqq 3\\ r+s=4
        \end{smallmatrix}$ & $\C^{k,\ell} \otimes_\C \C^{2r,2s},\,\,\,
        (4kr+4\ell s,4ks +4\ell r)$ & $\Im\C$ \\ \hhline{|~|-|-|-|}
   & $U(k,\ell) \cdot Sp(4;\R), k+\ell \geqq 3$ 
	& $\C^{k,\ell}\otimes_\C \C^8,\,\,\,
        (8m,8m)$ & $\Im\C$ \\ \hline\hline
{\rm 17} & $U(1) \cdot Spin(7)$ & $\C^8,\,\,\, (16,0)$
        & $\Im\C$ \\ \hhline{~|-|-|-}
   & $U(1) \cdot Spin(6,1)$ & $\C^{6,2},\,\,\, (12,4)$
        & $\Im\C$ \\ \hhline{~|-|-|-}
   & $U(1) \cdot Spin(5,2)$ & $\C^{6,2},\,\, (12,4)$
        & $\Im\C$ \\ \hhline{~|-|-|-}
   & $U(1) \cdot Spin(4,3)$ & $\C^{4,4},\,\, (8,8)$
        & $\Im\C$ \\ \hline\hline
{\rm 18} & $U(1) \cdot Spin(9)$ & $\C \otimes_\R \R^{16},\,\, (32,0)$
        & $\Im\C$ \\ \hhline{~|-|-|-}
   & $U(1)\cdot Spin(r,s), r+s=9$ & $\C^{8,8},\, (16,16)$
        & $\Im\C$ \\ \hline\hline
{\rm 19} & $Spin(10)$ & $\C^{16},\,\, (32,0)$ & $\Im\C$ \\ 
	\hhline{|~|-|-|-|}
%   & $Spin(9,1)$ & $\R^{16,16},\, (16,16)$
%        & $\Im\C$ \\ \hhline{~|-|-|-}
   & $Spin(8,2)$ & $\C^{8,8},\, (16,16)
        $ & $\Im\C$ \\
        \hhline{|~|-|-|-|}
%   & $Spin(7,3)$ & $\H^{4,4}$,\, (16,16)
%        & $\Im\C$ \\
%        \hhline{|~|-|-|-|}
   & $Spin(6,4)$ & $\C^{8,8},\, (16,16)$
        & $\Im\C$ \\
%        \hhline{|~|-|-|-|}
%   & $Spin(5,5)$ & $\R^{16,16},\, (16,16)$
%        & $\Im\C$ \\
         \hline\hline
{\rm 20} & $U(1) \cdot Spin(10)$ & $\C^{16},\,\, (32,0)$
        & $\Im\C$ \\ \hhline{|~|-|-|-|}
   & $U(1)\cdot Spin(8,2)$ & $\C^{8,8},\, (16,16)
        $ & $\Im\C$ \\
        \hhline{|~|-|-|-|}
   & $U(1)\cdot Spin(6,4)$ & $\C^{8,8},\, (16,16)$
        & $\Im\C$ \\
        \hhline{|~|-|-|-|}
   & $U(1)\cdot Spin^*(10)$ & $\H^{4,4},\, (16,16)$
        & $\Im\C$ \\
        \hline\hline
{\rm 21} & $U(1) \cdot G_2$ & $\C^7,\,\, (14,0)$ & $\Re\O$
        \\ \hhline{|~|-|-|-|}
   & $U(1)\cdot G_{2,A_1A_1}$ & $\C^{3,4},\, (6,8)$
        & $\Re\O_{sp}$ \\
        \hline\hline
{\rm 22} & $U(1) \cdot E_6$ & $\C^{27},\,\, (54,0)$ &
        $\Im\C$ \\ \hhline{|~|-|-|-|}
   & $U(1)\cdot E_{6,A_5A_1}$ & $\C^{15,12},\, (30,24)$
        & $\Im\C$ \\
        \hhline{|~|-|-|-|}
   & $U(1)\cdot E_{6,D_5T_1}$ & $\C^{16,11},\, (32,22)$
        & $\Im\C$ \\
\end{longtable}

In Table \ref{heis-comm}, every entry is contained in a ``maximal''
entry for which $\dim_\C\gv = m$ and $H$ is a subgroup of $U(m)$.
There are so many of those, that it is best to restrict attention
to the cases where the action of $H$ on $\gv$ is irreducible.
The next examples are those for which the action of $H$ on $\gv$ is
irreducible.  We extract them from \cite[Table 5.2]{WC2018}.
Again we retain the numbering corresponding to real form families from that 
table.  Also we omit the cases where $N$ is commutative, i.e. where
$G/H = \C^n$.

\addtocounter{equation}{1}
{\footnotesize
\begin{longtable}{|r|l|l|l|}
\caption*{\bf \qquad\qquad\quad {\normalsize Table} \thetable \quad {\normalsize
Maximal Irreducible Weakly Symmetric Nilmanifolds}  \\
\centerline{\normalsize $(N\rtimes H,H)$ of Complex Type }} \label{max-irred} \\
\hline
 & Group $H$ & $\gv \text{ and signature}(\gv)$&
                $\gz$  \\ \hline
\hline
\hline
\endfirsthead
\multicolumn{4}{l}{{\normalsize \textit{Table \thetable\, continued from
        previous page $ \dots$}}} \\
\hline
 & Group $H$ & $\gv \text{ and signature}(\gv)$&
                $\gz$  \\ \hline
\hline
\endhead
\hline \multicolumn{4}{r}{{\normalsize \textit{$\dots$ Table \thetable\,
        continued on next page}}} \\
\endfoot
\hline
\endlastfoot
\hline
4 & $U(1)\cdot SO(r,s), r+s\ne 4$ & $\C^{r,s},\,\,(2r,2s)$ 
	& $\Im \C$ \\
        \hhline{|~|-|-|-|}
  & $U(1)\cdot SO^*(n), n=2m\ne 4$ & $\C^{m,m},\,\,(2m,2m)$
        & $\Im \C$  \\ \hline \hline

5 & $SU(r,s)$, $r+s$ even & $\C^{r,s},\,\,(2r,2s)$
                & $\Lambda_\R^2(\C^{r,s}) \oplus\Im\C$
         \\ \hhline{|~|-|-|-|}
  & $U(r,s)$ & $\C^{r,s},\,\, (2r,2s)$
        & $\Lambda_\R^2(\C^{r,s}) \oplus\Im\C$
         \\ \hline\hline

6 & $SU(r,s), r+s$ odd & $\C^{r,s},\,\, (2r,2s)$
        & $\Lambda_\R^2(\C^{r,s})$
         \\ \hline\hline

7 & $SU(r,s), r+s$ odd & $\C^{r,s},\,\, (2r,2s)$ & $\Im\C$   
	\\ \hline\hline

8 & $U(r,s)$ & $\C^{r,s},\,\, (2r,2s)$ & $\gu(r,s)$  
	\\ \hline\hline

9 & $(\{1\} \text{ or } U(1))\cdot Sp(r,s)$ & $\H^{r,s},\,\, (4r,4s)$
        & $\Re \H^{(r,s) \times (r,s)}_0 \oplus 
                \Im\H$  \\ \hhline{|~|-|-|-|}
   & $U(1)\cdot Sp(n;\R)$ & $\R^{2n,2n},\,\, (2n,2n)$
        & $\Re \H_{sp,0}^{n \times n} \oplus 
               \Im\H_{sp}$ 
	\\ \hline\hline

10 & $U(r,s)$ & $S_\C^2(\C^{r,s})$,\,\, {\tiny $(r(r+1) + s(s+1), 2rs)$}
        & $\Im\C$  
	\\ \hline\hline

11 & $SU(r,s), r+s \geqq 3, r+s$ odd & $\Lambda^2_\C(\C^{r,s}),\,\, 
        (r^2-r+s^2-s,2rs)$ & $\Im\C$ \\ \hhline{|~|-|-|-|}
   & $U(r,s), r+s \geqq 3$ &  ${\Lambda}^2_\C(\C^{r,s}),\,\, 
        (r^2-r+s^2-s,2rs)$ & $\Im\C$
	\\ \hline\hline

12 & $U(1)\cdot Spin(7)$ & $\O_\C = \C\otimes_\R \R^8,\,\, (16,0)$
        & $\R^7 \oplus \R$ \\ \hhline{|~|-|-|-|}
  & $U(1)\cdot Spin(6,1)$ & $\C \otimes_\R \R^{6,2},\,\, (12,4)$
        & $\R^{6,1}\oplus \R$ \\
        \hhline{|~|-|-|-|}
  & $U(1)\cdot Spin(5,2)$ & $\C \otimes_\R \R^{6,2},\,\, (12,4)$
        & $\R^{5,2}\oplus \R$ \\
        \hhline{|~|-|-|-|}
  & $U(1)\cdot Spin(4,3)$ & $\C \otimes_\R \R^{4,4},\,\, (8,8)$
        & $\R^{4,3}\oplus \R$ 
	\\ \hline\hline

13 & $U(1)\cdot Spin(9)$ & $\C\otimes_\R \R^{16},\,\, (32,0)$
        & $\R$ \\ \hhline{|~|-|-|-|}
   & $U(1)\cdot Spin(8,1)$ & $\C\otimes_\R \R^{8,8},\,\, (16,16)$
        & $\Im\C$ \\ \hhline{|~|-|-|-|}
   & $U(1)\cdot Spin(7,2)$ & $\C\otimes_\R \C^{4,4},\,\, (16,16)$
        & $\Im\C$ \\ \hhline{|~|-|-|-|}
   & $U(1)\cdot Spin(6,3)$ & $\C\otimes_\R \C^{4,4},\,\, (16,16)$
        & $\Im\C$ \\ \hhline{|~|-|-|-|}
   & $U(1)\cdot Spin(5,4)$ & $\C\otimes_\R \H^{2,2},\,\, (16,16)$
        & $\Im\C$ 
	\\ \hline\hline

14 & $(\{1\} \text{ or } U(1)) \cdot Spin(10)$ & $\C^{16},\,\, (32,0)$
        & $\Im\C$ \\ \hhline{|~|-|-|-|}
   & $Spin(9,1)$ & $\R^{16,16},\,\, (16,16)$ & $\R$ \\
        \hhline{|~|-|-|-|}
   & $(\{1\} \text{ or } U(1)) \cdot Spin(8,2)$ & $\C^{8,8},\,\, (16,16)$
        & $\Im\C$ \\ \hhline{|~|-|-|-|}
   & $Spin(7,3)$ & $\H^{4,4}$,\,\, (16,16) & $\R$ \\
        \hhline{|~|-|-|-|}
   & $(\{1\} \text{ or } U(1)) \cdot Spin(6,4)$ & $\C^{8,8},\,\, (16,16)$
        & $\Im\C$ \\ \hhline{|~|-|-|-|}
   & $Spin(5,5)$ & $\R^{16,16},\,\, (16,16)$ & $\R$ \\
        \hhline{|~|-|-|-|}
   & $U(1)\cdot Spin^*(10)$ & $\H^{4,4},\,\, (16,16)$ & $\Im\C$ 
	\\ \hline\hline

15 & $U(1)\cdot G_2$ & $\C^7 = \Im\O_\C,\,\, (14,0)$
        & $\R = \Re\O$ \\ \hhline{|~|-|-|-|}
   & $U(1)\cdot G_{2,A_1A_1}$ & $\C^{3,4},\,\, (6,8)$
        & $\Re\O_{sp}$ 
	\\ \hline\hline

16 & $U(1)\cdot E_6$ & $\C^{27},\,\, (54,0)$
        & $\Im\C$ \\ \hhline{|~|-|-|-|}
   & $U(1)\cdot E_{6,A_5A_1}$ & $\C^{15,12},\,\, (30,24)$
        & $\Im\C$ \\ \hhline{|~|-|-|-|}
   & $U(1)\cdot E_{6,D_5T_1}$ & $\C^{16,11},\,\, (32,22)$
        & $\Im\C$ 
	\\ \hline\hline

19 & $\begin{matrix} (\{1\} \text{ or }U(1))\cdot (SU(k,\ell) \cdot SU(r,s)), \\
        k+\ell,r +s\geqq 3, U(1) \text{ if } k+\ell = r+s \end{matrix}$
        & {\tiny $\C^{(k,\ell)\times (r,s)},\,\, (2kr+2\ell s, 2ks + 2\ell r)$}
        & $\Im\C$     \\ \hhline{|~|-|-|-|}
  & $(\{1\} \text{ or }U(1))\cdot SL(m;\C)$ & $\gg\gl(m;\C),\,\, (m^2,m^2)$
        & $\Im\C$ \\ \hline\hline

20 & $\begin{matrix} (\{1\} \text{ or } U(1)) \cdot (SU(2) \cdot SU(r,s)), 
        \\ r+s \geqq 2, U(1) \text{ if } r+s=2 \end{matrix}$
        & $\C^{2\times (r,s)},\,\, (4r,4s)$
        & $\Im \C^{2\times 2} = \gu(2)$   \\ \hhline{|~|-|-|-|}
    & $(\{1\} \text{ or } U(1)) \cdot(SU(1,1) \cdot SU(r,s))$
        & $\C^{(1,1)\times (r,s)},\,\, (2r+2s,2r+2s)$ & $\gu(1,1)$     
	\\ \hline\hline

21 & $\begin{matrix} (\{1\} \text{ or } U(1))\cdot (Sp(2) \cdot SU(r,s)), 
	\\ r+s \geqq 3, U(1) \text{ if } r+s \leqq 4 \end{matrix}$
        & $\H^2\otimes_\R \C^{r,s},\,\, (16r,16s)$ & $\Im\C$  \\ 
	\hhline{|~|-|-|-|}
   & $(\{1\} \text{ or } U(1)) \cdot (Sp(1,1) \cdot SU(r,s))$
        & $\H^{1,1}\otimes_\R \C^{r,s},\,\, (8r+8s,8r+8s)$ & $\Im\C$  \\
        \hhline{|~|-|-|-|}
   & $Sp(2;\R) \cdot U(r,s))$
        & $\R^{4,4}\otimes_\R \C^{r,s} ,\,\, (8r+8s,8r+8s)$ & $\Im\C$
        \\ \hline\hline

22 & $H = U(k,\ell)\cdot Sp(r,s), k+\ell =2$ & {\tiny $\C^{k,\ell} \otimes_\C \H^{r,s},\,\, 
        (4kr+4\ell s, 4ks+4\ell r)$} &
        $\gu(k,\ell)$   \\ \hhline{|~|-|-|-|}
   & $H = U(k,\ell)\cdot Sp(n;\R), k+\ell =2$ & {\tiny $\C^{k,\ell}\otimes_\C 
        \otimes_\C (\R^{2n}\oplus\R^{2n}),\, (4n,4n)$}
        & $\gu(k,\ell)$   
	\\ \hline\hline

23 & $H = U(k,\ell)\cdot Sp(r,s),\, k+\ell=3$
        & {\tiny $\C^{k,\ell} \otimes_\C \H^{r,s},\,\, 
        (4kr+4\ell s, 4ks+4\ell r)$}
        & $\Im\C$   \\ \hhline{|~|-|-|-|}
   & $H = U(k,\ell)\cdot Sp(n;\R),\, k+\ell=3$
        & $\C^{k,\ell} \otimes_\C \C^{2n},\,\, (6n,6n)$
        & $\Im\C$   \\ \hline
\end{longtable}
}

\section{Toward the Classification for Weakly Symmetric Pseudo--Riemannian
Nilmanifolds of Complex Type} \label{sec6}
\setcounter{equation}{0}

The classification of irreducible to indecomposable commutative spaces
is due to Yakimova.  It is combinatorial, based on her classification 
(\cite{Y2005}, \cite{Y2006}; or see \cite{W2007}) of indecomposable
commutative spaces --- subject to a few technical conditions.
In this section we broaden the scope of Table \ref{max-irred} from
irreducible to indecomposable commutative spaces, subject to those
technical conditions.  The technical conditions, which we explain
just below, are that $(N\rtimes H,H)$ be indecomposable, principal, 
maximal and $Sp(1)$--saturated.  

We work out the classification of weakly symmetric pseudo--riemannian
nilmanifolds of complex type for the real form families corresponding
to those indecomposable commutative spaces.  This is the main 
non--combinatorial step in classifying all the weakly symmetric 
pseudo--riemannian nilmanifolds of complex type.

Since $G = N\rtimes H$ acts almost--effectively on $M = G/H$, the
centralizer of $N$ in $H$ is discrete, in other words the representation
of $H$ on $\gn$ has finite kernel.  (In the notation of
\cite[Section 1.4]{Y2006} this says $H = L = L^0$ and $P = \{1\}$.)
That simplifies the general definitions \cite[Definition 6]{Y2006} of
{\bf principal} and \cite[Definition 8]{Y2006} of {\bf $Sp(1)$--saturated},
as follows.  Decompose $\gv$ as a sum $\gw_1 \oplus \dots \oplus \gw_t$ of
irreducible $\Ad(H)$--invariant subspaces.  Then $(G,H)$ is {\bf principal}
if $Z_H^0 = Z_1 \times \dots \times Z_m$ where $Z_i \subset GL(\gw_i)$,
in other words $Z_i$ acts trivially on $\gw_j$ for $j \ne i$.
Decompose $H = Z_H^0 \times H_1 \times \dots \times H_m$ where the $H_i$ are
simple.  Suppose that whenever some $H_i$ acts nontrivially on some $\gw_j$
and $Z_H^0 \times \prod_{\ell \ne i} H_\ell$ is irreducible on $\gw_j$, it
follows that $H_i$ is trivial on $\gw_k$ for all $k \ne j$.  Then
$H_i \cong Sp(1)$ and we say that $(G,H)$ is {\bf $Sp(1)$--saturated}.
The group $Sp(1)$ will be more visible in the definition when we
extend the definition to the cases where $H \ne L$.

In the following table, $\gh_{n;\F}$ is the Heisenberg algebra
$\Im\F + \F^n$ of real dimension $(\dim_\R\F - 1)+ n\dim_\R\F$.  Here
$\F$ is the real, complex, quaternion or octonion algebra over $\R$,
$\Im\F$ is its imaginary component, and
$$
\gh_{n;\F} = \Im \F + \F^n \text{ with product }
        [(z_1,v_1),(z_2,v_2)] = (\Im(v_1\cdot v_2^*),0)
$$
where the $v_i$ are row vectors and $v_2^*$ denotes the conjugate ($\F$
over $\R$) transpose of $v_2$\,.  It is the
Lie algebra of the (slightly generalized) Heisenberg group $H_{n;\F}$\,.
Also in the table, in the listing for $\gn$
the summands in double parenthesis ((..)) are the subalgebras
$[\gw,\gw] + \gw$ where $\gw$ is an $H$--irreducible subspace of $\gv$
with $[\gw,\gw] \ne 0$, and the summands not in double parentheses are
$H$--invariant subspaces $\gw \subset \gz$ with $[\gw,\gw] = 0$.  Thus
\begin{equation} \label{bigger-center}
\gn = \gz + \gv \text{ vector space direct sum, and } \gz = [\gn,\gn] \oplus \gu
\end{equation}
where the center $\gu$ is the sum of the summands listed for $\gn$ that are 
{\it not} enclosed in double parenthesis ((..)).

As before, when we write $m/2$ it is assumed that
$m$ is even, and similarly $n/2$ requires that $n$ be even.   Further
$k+\ell = m$ and $r+s = n$ where applicable.

\addtocounter{equation}{1}
{\footnotesize
\begin{longtable}{|r|l|l|l|}
\caption*{\bf \small
Table \thetable \quad Maximal Indecomposable Principal
Commutative Nilmanifolds $(N\rtimes H,H)$ \\
\centerline{$N$ Nonabelian, Where the Action of $H$ on
$\gn/[\gz,\gz]$ is Reducible}} \label{indecomp} \\
\hline
 & \begin{tabular}{c} Group $H$ and\\ Algebra $\gn$ \end{tabular}
        & \begin{tabular}{c} $H$--module $\gv$ and \\ 
		Signature$(\gv)$ \end{tabular}
        & \begin{tabular}{c} $[\gn,\gn]$ \\ $\gu$ \end{tabular} \\
\hline
\hline
\endfirsthead
\multicolumn{4}{l}{{\normalsize \textit{Table \thetable\, continued from
        previous page $\dots$}}} \\ \hline
 & \begin{tabular}{c} Group $H$ and\\ Algebra $\gn$ \end{tabular}
        & \begin{tabular}{c} $H$--module $\gv$ and \\ 
                Signature$(\gv)$ \end{tabular}
        & \begin{tabular}{c} $[\gn,\gn]$ \\ $\gu$ \end{tabular} \\
\hline
\hline
\endhead
\hline \multicolumn{4}{r}{{\normalsize \textit{$\dots$ Table \thetable\,
        continued on next page}}} \\
\endfoot
\hline
\endlastfoot

1 & \begin{tabular}{l} $U(r,s)$\\
        $((\gh_{r+s;\C}))+\gs\gu(r,s)$ \end{tabular}
        & \begin{tabular}{l} $\C^{r,s}$ \\ $(2r,2s)$ \end{tabular}
        & \begin{tabular}{l} $\Im\C$ \\ $\gs\gu(r,s)$ \end{tabular}
	\\ \hline\hline

2 & \begin{tabular}{l} $U(r,s), (r,s)=(4,0) \text{ or } (2,2)$\\
	 $((\Im\C + \Lambda^2(\C^{r,s}) +\C^{r,s}))+
		\Lambda^2(\R^{r,s})$\end{tabular} 
	& \begin{tabular}{l} $\C^{r,s}$ \\ $(2r,2s)$ \end{tabular}
        & \begin{tabular}{l} $\Lambda^2(\C^{r,s}) + \Im\C$ \\
		$\Lambda^2(\R^{r,s})$	\end{tabular} 
        \\ \hline\hline

3 & \begin{tabular}{l} $U(1)\cdot SU(r,s)\cdot U(1)$\\
	 $((\gh_{n;\C}))+((\gh_{n(n-1)/2;\C}))$ \end{tabular}
        & \begin{tabular}{l}  $\C^{r,s}\oplus\Lambda_\C^2(\C^{r,s})$ \\
                $(2r,2s)$\\
		\,\,$\oplus (r^2-r+s^2-s,2rs)$ \end{tabular}
        & \begin{tabular}{l} $\Im\C\oplus\Im\C$\\ $\{0\}$ \end{tabular} 
	\\ \hline\hline

4 & \begin{tabular}{l} $SU(r,s),\, (r,s) = (4,0) \text{ or } (2,2)$ \\
	 $((\Im\C + \Re\H^{2\times 2} + \C^{r,s})) + 
		\Lambda^2(\R^{r,s})$\end{tabular}
        & \begin{tabular}{l} $\C^{r,s}$ \\
	 $(2r,2s)$ \end{tabular} 
        & \begin{tabular}{l} $\Im\C \oplus \Re\H^{2\times 2}$ \\
		$\Lambda^2(\R^4)$\end{tabular}
	\\ \hline\hline

5 & \begin{tabular}{l} $U(k,\ell)\times U(r,s)$ \\
		$k+\ell = 2, (r,s) = (4,0) \text{ or } (2,2)$ \\
         $((\Im \C^{(k,\ell)\times (k,\ell)} + \C^{(k,\ell)\times (r,s)})) + 
		\Lambda^2(\R^{r,s})$\end{tabular}
        & \begin{tabular}{l} $\C^{(k,\ell)\times (r,s)}$ \\
		 $(2kr+2\ell s, 2ks + 2\ell r)$ \end{tabular}
        & \begin{tabular}{l} $\Im \C^{(k,\ell)\times (k,\ell)}$ \\
		$\Lambda^2(\R^{r,s})$ \end{tabular}
	\\  \hline\hline

6 & \begin{tabular}{l} $S(U(k,\ell)\times U(r,s)), 
		(k,\ell) = (4,0) \text{ or } (2,2)$ \\
		$((\gh_{4(r+s);\C})) + \Lambda^2(\R^{k,\ell})$ \end{tabular}
        & \begin{tabular}{l} $\C^{(k,\ell)\times(r,s)}$ \\
		 $(2kr+2\ell s, 2ks + 2\ell r)$ \end{tabular}
        & \begin{tabular}{l} $\Im\C$ \\ $\Lambda^2(\R^{k,\ell})$ \end{tabular}
	\\  \hline\hline

7 & \begin{tabular}{l} $U(k,\ell) \cdot U(r,s)$\\ 
	$((\gh_{(k+\ell,r+s);\C})) + ((\gh_{k+\ell;\C}))$ \end{tabular}
        & \begin{tabular}{l} $\C^{(k,\ell)\times (r,s)}\oplus\C^{k,\ell}$ \\
                $(2kr+2\ell s,2ks+2\ell r)$ \\
		\,\,\,$\oplus (2k,2\ell)$ \end{tabular}
        & \begin{tabular}{l} $\Im\C\oplus\Im\C$\\ $\{0\}$ \end{tabular} 
        \\ \hline\hline

8 & \begin{tabular}{l} $U(1)\cdot Sp(r,s)\cdot U(1)$\\ 
        	$((\gh_{2(r+s);\C})) + ((\gh_{2(r+s);\C}))$ \end{tabular} 
        & \begin{tabular}{l} $\C^{2r,2s}\oplus \C^{2r,2s}$ \\
		$(4r,4s)\oplus (4r,4s)$ \end{tabular}
        & \begin{tabular}{l} $\Im\C\oplus\Im\C$\\ $\{0\}$\end{tabular}
	 \\  \hhline{|~|-|-|-|}
  & \begin{tabular}{l} $U(1) \cdot Sp(n;\R)\cdot U(1)$ \\
	  $((\gh_{2n;\C})) \oplus ((\gh_{2n;\C}))$ \end{tabular}
        & \begin{tabular}{l} $\C^{n,n} \oplus \C^{n,n}$\\
	 $(2n,2n)\oplus (2n,2n)$ \end{tabular}
        &\begin{tabular}{l} $\Im\C\oplus\Im\C$\\ $\{0\}$\end{tabular} 
	\\ \hline\hline

11 & \begin{tabular}{l} $Sp(k,\ell)\cdot
	(U(1)\text{ or }\{1\})\cdot Sp(r,s)$ \\
           $((\gh_{k+\ell;\H})) + \H^{(k,\ell)\times (r,s)}$ \end{tabular}
        & \begin{tabular}{l}$\H^{k,\ell}$ \\ $(4k,4\ell)$ \end{tabular}
	& \begin{tabular}{l} $\Im\H = \gs\gp(1)$\\ 
		$\H^{(k,\ell)\times (r,s)}$ \end{tabular}
        \\  \hhline{|~|-|-|-|}
   & \begin{tabular}{l} $Sp(m;\R)\cdot (U(1) \text{ or }\{1\})\cdot Sp(n;\R)$\\
	  $((\gh_{m;\H})) + \C^{(m,m)\times(n,n)}$ \end{tabular}
        & \begin{tabular}{l} $\C^{m,m}$ \\ $(2m,2m)$ \end{tabular}
	& \begin{tabular}{l} $\gs\gp(1;\R)$ \\ $\C^{(m,m)\times(n,n)}$
		\end{tabular}
        \\ \hline\hline

12 & \begin{tabular}{l} $Sp(k,\ell)\cdot(U(1)\text{ or }\{1\})$  \\
        $((\gh_{k+\ell;\H})) + 
		\Re \H^{(k,\ell)\times (k,\ell)}_0$ \end{tabular}
	& \begin{tabular}{l} $\H^{k,\ell}$ \\ $(4k,4\ell)$ \end{tabular}
	& \begin{tabular}{l} $\Im\H = \gs\gp(1)$ \\ 
		$\Re \H^{(k,\ell)\times (k,\ell)}_0$ \end{tabular}
	\\  \hhline{|~|-|-|-|}
   & \begin{tabular}{l} $Sp(m;\R)\cdot (U(1)\text{ or } \{1\})$ \\
		$((\gh_{m;\H})) + \Re \H^{m\times m}_{sp,0}$\end{tabular}
	& \begin{tabular}{l} $\C^{m,m}$ \\ $(2m,2m)$ \end{tabular} 
	& \begin{tabular}{l} $\Im\H$ \\ $\Re \H^{m\times m}_{sp,0}$\end{tabular}
	\\ \hline\hline

13 & \begin{tabular}{l} $Spin(k,7-k)\cdot (SO(2)\text{ or } \{1\})$ \\
         $((\gh_{1;\O})) + \R^{(k,7-k)\times 2},\, 4\leqq k\leqq 7$ 
		\end{tabular}
        & \begin{tabular}{l} $\R^{q,8-q},\, q=2[\tfrac{k+1}{2}]$ 
		\\ $(q,8-q)$ \end{tabular}
        & \begin{tabular}{l} $\R^{k,7-k}$\\ $\R^{(k,7-k)\times 2}$\end{tabular} 
	\\ \hline\hline

14 & \begin{tabular}{l} $U(1)\cdot Spin(k,7-k),\, 4\leqq k\leqq 7$ \\
         $((\gh_{7;\C})) + \R^{q,8-q},\, q=2[\tfrac{k+1}{2}]$ \end{tabular}
        & \begin{tabular}{l} $\C^{2k,14-2k}$ \\ $(2k,14-2k)$ \end{tabular}
        & \begin{tabular}{l} $\Im\C$ \\ $\R^{q,8-q}$ \end{tabular}
        \\ \hline\hline

15 & \begin{tabular}{l} $U(1)\cdot Spin(k,7-k),\, 4\leqq k\leqq 7$ \\ 
	  $((\gh_{8;\C})) + \R^7$ \end{tabular}
        & \begin{tabular}{l} $\C^{q,8-q},\, q=2[\tfrac{k+1}{2}]$ \\
		$(2q,16-2q)$ \end{tabular}
        & \begin{tabular}{l} $\Im\C$ \\ $\R^{k,7-k}$ \end{tabular}
	\\  \hline\hline

16 & \begin{tabular}{l} $U(1)\cdot Spin(k,8-k)\cdot U(1)$ \\
		$((\gh_{8;\C})) + ((\gh_{8;\C}))$ \end{tabular}
        & \begin{tabular}{l} $\C^{k,\ell}_+ \oplus\C^{k,\ell}_-$ \\
		$(2k,2\ell)\oplus(2k,2\ell)$ \end{tabular}
        & \begin{tabular}{l} $\Im\C \oplus\Im\C$ \\  $\{0\}$ \end{tabular}
	\\ \hhline{|~|-|-|-|}
   & \begin{tabular}{l} $U(1)\cdot Spin^*(8)\cdot U(1)$ \\
		$((\gh_{8;\C})) + ((\gh_{8;\C}))$ \end{tabular}
        & \begin{tabular}{l} $\C^{4,4}\oplus\C^{4,4}$\\
		$(8,8)\oplus(8,8)$ \end{tabular}
        & \begin{tabular}{l} $\Im\C\oplus\Im\C$ \\ $\{0\}$ \end{tabular} 
	\\  \hline\hline

17 & \begin{tabular}{l} $U(1)\cdot Spin(2k,2\ell),\, 
		\begin{smallmatrix} k=3,4,5\\ \ell=5-k\end{smallmatrix}$\\
		$((\gh_{16;\C})) + \R^{2k,2\ell}$ \end{tabular}
        & \begin{tabular}{l} $\C^{q,16-q},\, q= 2^{[\frac{k+1}{2}]+2}$\\
		$(2q, 32-2q)$ \end{tabular}
	& \begin{tabular}{l} $\Im\C$ \\ $\R^{2k,2\ell}$ \end{tabular}
        \\ \hhline{|~|-|-|-|}
   & \begin{tabular}{l} $U(1)\cdot Spin^*(10)$ \\
	$((\gh_{16;\C})) + \R^{10}$ \end{tabular}
        & \begin{tabular}{l} $\C^{8,8}$\\$(16,16)$\end{tabular}
        & \begin{tabular}{l} $\Im\C$ \\ $\R^{10}$ \end{tabular}
	\\  \hline\hline

18 & \begin{tabular}{l} 
	$(SU(k,\ell)\text{ or } U(k,\ell)\text{ or }$ 
		$U(1)Sp(\tfrac{m}{2})) \cdot SU(r,s)$ \\
                \, $k+\ell=m, r+s=2$ \\
		$((\gh_{2m;\C})) + \gs\gu(r,s)$\end{tabular}
        & \begin{tabular}{l}$\C^{(k,\ell)\times (r,s)}$ \\
		$(2kr+2\ell s, 2ks+2\ell r)$\end{tabular}
        & \begin{tabular}{l} $\Im\C$ \\ $\gs\gu(r,s)$ \end{tabular}
	\\  \hhline{|~|-|-|-|}
    & \begin{tabular}{l} $Sp(m/2;\R)\cdot U(r,s)$ \\
	$((\gh_{2m;\C})) + \gs\gu(r,s)$\end{tabular}
        & \begin{tabular}{l} $\C^{(m/2,m/2)\times (r,s)}$\\
                $(2m,2m)$ \end{tabular}
        & \begin{tabular}{l} $\Im\C$ \\ $\gs\gu(r,s)$ \end{tabular}
	\\ \hline\hline

19 & \begin{tabular}{l}
        $(SU(k,\ell)\text{ or } U(k,\ell)\text{ or}$  
                $U(1)Sp(\tfrac{m}{2})) \cdot U(r,s)$ \\
                $k+\ell=m, r+s=2$ \\ 
                $((\gh_{2m;\C})) + ((\gh_{2;\C}))$\end{tabular}
        & \begin{tabular}{l}$\C^{(k,\ell)\times (r,s)}\oplus \C^{r,s}$ \\
                $(2kr+2\ell s, 2ks+2\ell r)$ \\
		\qquad $\oplus (2r,2s)$\end{tabular}
        & \begin{tabular}{l} $\Im\C \oplus \Im\C$ \\ $\{0\}$ \end{tabular}
	\\  \hhline{|~|-|-|-|}
    & \begin{tabular}{l} $U(1)Sp(\tfrac{m}{2};\R)\cdot U(r,s)$ \\
	$((\gh_{2m;\C})) + ((\gh_{2;\C}))$\end{tabular}	
        & \begin{tabular}{l} $\C^{(\tfrac{m}{2},\tfrac{m}{2})
		\times (r,s)}\oplus\C^{r,s}$ \\
                 $(2m,2m)\oplus (2r,2s)$ \end{tabular}
        & \begin{tabular}{l} $\Im\C\oplus\Im\C$ \\ $\{0\}$ \end{tabular}
	\\ \hline \hline

20 & {\tiny \begin{tabular}{l} $(SU(k,\ell), U(k,\ell), 
			U(1)Sp(\tfrac{k}{2},\tfrac{\ell}{2}))$
		$\cdot SU(a,b)$ \\
                $\quad\cdot (SU(r,s), U(r,s), 
			U(1)Sp(\tfrac{r}{2},\tfrac{s}{2}))$\\
		$k+\ell = m, a+b = 2, r+s = n$ \\
		$((\gh_{2m;\C})) + ((\gh_{2n;\C}))$ \\
		\end{tabular}}
        & \begin{tabular}{l}  $\C^{(k,\ell)\times (a,b)}\oplus 
                \C^{(a,b)\times (r,s)}$ \\
                $(2(ak+b\ell),2(a\ell +bk))$ \\
                $\,\oplus (2(ar+bs),2(as+br))$
                \end{tabular}
        & \begin{tabular}{l} $\Im\C\oplus\Im\C$ \\ $\{0\}$ \end{tabular}
	\\ \hhline{|~|-|-|-|}
    & {\tiny \begin{tabular}{l} $(SU(k,\ell), U(k,\ell),
		U(1)Sp(\tfrac{k}{2},\tfrac{\ell}{2}))$ \\
                $\quad\cdot SU(a,b)\cdot U(1)Sp(\tfrac{n}{2};\R)$ \\
 		$k+\ell = m, a+b = 2, r+s = n$\\
		$((\gh_{2m;\C})) + ((\gh_{2n;\C}))$ \\
		\end{tabular}}
        & {\tiny \begin{tabular}{l} $\C^{(k,\ell)\times (a,b)}\oplus 
                \C^{(a,b) \times(\tfrac{n}{2},\tfrac{n}{2})}$
                \\ $(2(ak+b\ell),2(a\ell +bk))\oplus 
		(\tfrac{n}{2},\tfrac{n}{2})$ 
		\end{tabular}}
        & \begin{tabular}{l} $\Im\C\oplus\Im\C$ \\ $\{0\}$ \end{tabular}
	\\ \hhline{|~|-|-|-|}
     & {\tiny \begin{tabular}{l} $U(1)Sp(\tfrac{m}{2};\R)\cdot SU(a,b)
		\cdot U(1)Sp(\tfrac{n}{2};\R)$\\
		$k+\ell = m, a+b = 2, r+s = n$ \\
		$((\gh_{2m;\C})) + ((\gh_{2n;\C}))$ \\
		\end{tabular}}
        &  \begin{tabular}{l} $\C^{(m/2,m/2) \times (a,b)}$\\
			$\qquad \oplus \C^{(a,b) \times(n/2,n/2)}$ \\
                $(2m,2m)\oplus (2n,2n)$ 
		\end{tabular}
        & \begin{tabular}{l} $\Im\C\oplus\Im\C$ \\ $\{0\}$ \end{tabular}
	\\ \hline\hline

21 & {\tiny \begin{tabular}{l} $(SU(k,\ell), U(k,\ell), 
                        U(1)Sp(\tfrac{k}{2},\tfrac{\ell}{2}))$\\
                $\quad \cdot SU(a,b)\cdot U(r,s), r,s \text{ even}$\\
	$k+\ell = m, a+b = 2, r+s=4$ \\
                $((\gh_{2m;\C})) + ((\gh_{8;\C})) + \Lambda^2(\R^{r,s})$ \\
                \end{tabular}}
        & {\tiny \begin{tabular}{l}  $\C^{(k,\ell)\times (a,b)}\oplus 
                \C^{(a,b)\times (r,s)}$ \\
                $(2(ak+b\ell),2(a\ell +bk))$ \\
                $\,\,\oplus (2(ar+bs),2(as+br))$
                \end{tabular}}
        & \begin{tabular}{l}$\Im\C \oplus \Im\C$ \\ $\Lambda^2(\R^{r,s})$ 
		\end{tabular}
	\\ \hhline{|~|-|-|-|}
    & {\tiny \begin{tabular}{l} $U(1)Sp(m;\R)\cdot SU(a,b))\cdot U(r,s))$ \\
	$a+b=2, (r,s)=(4,0) \text{ or } (2,2)$ \\
	$((\gh_{2m;\C})) + ((\gh_{8;\C})) + \Lambda^2(\R^{r,s})$
	\end{tabular}}
        & \begin{tabular}{l} $\C^{(m,m)\times (a,b)} 
		\oplus \C^{(a,b)\times (r,s)}$ \\ 
		$\begin{smallmatrix} (4m,4m) \oplus (2(ar+bs),2(as+br))\\
			 \end{smallmatrix}$\end{tabular}
        & \begin{tabular}{l}$\Im\C \oplus \Im\C$ \\ $\Lambda^2(\R^{r,s})$ 
                \end{tabular}
	\\ \hline\hline

22 & {\tiny \begin{tabular}{l} $U(a,b)\cdot U(r,s)$ \\
 	$a+b=2, (r,s)=(4,0) \text{ or } (2,2)$ \\
         $((\gh_{8;\C})) + \Lambda^2(\R^{r,s}) +\gs\gu(2)$ \end{tabular}}
        & \begin{tabular}{l} $\C^{(a,b)\times (r,s)} $\\
	 $\,(ar+bs,as+br)$\end{tabular}
        & \begin{tabular}{l} $\Im\C$\\ {\tiny $\Lambda^2(\R^{r,s}) +\gs\gu(2)$}
		\end{tabular}
	\\ \hline\hline

23 & {\tiny \begin{tabular}{l} $U(k,\ell)\cdot U(a,b)\cdot U(r,s)$ \\
		$(k,\ell) = (4,0) \text{ or } (2,2),  a+b=2,$ \\
		$\qquad (r,s) = (4,0) \text{ or } (2,2)$\\
                 $\Lambda^2(\R^{k,\ell}) + ((\gh_{8;\C}))$ \\
		$\qquad + ((\gh_{8;\C})) + \Lambda^2(\R^{r,s})$ 
		\end{tabular}}
        & \begin{tabular}{l} $\C^{(k,\ell)\times (a,b)} 
		\oplus \C^{(a,b)\times (r,s)}$ \\
		$(ak+b\ell,a\ell +bk)$ \\
		  $\quad \oplus (ar+bs,as+br)$ 
                \end{tabular}
        & \begin{tabular}{l} $\Im\C\oplus\Im\C$ \\
		$\Lambda^2(\R^{k,\ell})$ \\  
		$\quad +\Lambda^2(\R^{r,s})$ \end{tabular}
	\\  \hline\hline

24 & \begin{tabular}{l} $U(1)\cdot SU(k,\ell)\cdot U(1)$ \\
		$(k,\ell) = (4,0) \text{ or } (2,2)$ \\
              $((\gh_{4;\C}))+((\gh_{4;\C})) + \Lambda^2(\R^{k,\ell})$
		\end{tabular}
        & \begin{tabular}{l}$\C^{k,\ell}\oplus\C^{k,\ell}$\\
		$(2k,2\ell)\oplus(2k,2\ell)$\end{tabular}
        & \begin{tabular}{l} $\Im\C\oplus\Im\C$ \\
		$\Lambda^2(\R^{k,\ell})$ \end{tabular}
	\\  \hline\hline

25 & {\tiny \begin{tabular}{l} $(U(1), \{1\}))\cdot
		SU(k,\ell))\cdot (\{1\}, U(1))$, \\
        $k+\ell = 4$ \\
	$((\gh_{4;\C})) + \Lambda^2(\C^{k,\ell})$ \end{tabular}}
        & \begin{tabular}{l} $\C^{k,\ell}$\\
		$(2k,2\ell)$\end{tabular}
        & \begin{tabular}{l} $\Im\C$ \\ $\Lambda^2(\C^{k,\ell})$ \end{tabular}
	\\   \hline
\end{longtable}
}

\pagebreak

\end{document}